\documentclass{amsart}
\usepackage{amsmath}
\usepackage{amsfonts}

\newtheorem{theorem}{Theorem}
\theoremstyle{plain}

\numberwithin{equation}{section}
\input{tcilatex}

\begin{document}
\title[Dual Fibonacci Quaternions]{Dual Fibonacci Quaternions}
\author{Semra KAYA NURKAN}
\address{ }
\email{}
\urladdr{}
\author{\.{I}lkay ARSLAN G\"{U}VEN}
\curraddr{ }
\email{ }
\urladdr{}
\date{}

\begin{abstract}
In this study, we define the dual Fibonacci quaternion and the dual Lucas
quternion. We derive the relations between the dual Fibonacci and the dual
Lucas quaternion which connected the Fibonacci and the Lucas numbers.
Furthermore, we give the Binet and Cassini formulas for these quaternions.
\end{abstract}

\maketitle

\section{\textbf{Introduction}}

The quaternions are a number system which extends to the complex numbers.
They are members of noncommutative algebra , first invented by William Rowan
Hamilton in 1843. Hamilton defined a quaternion as the quotient of two
vectors. The algebra of quaternions is denoted by$\mathbb{\ H}$ , also by
the Clifford algebra classifications $Cl_{0,2}(R)\cong Cl_{3,0}^{0}(R).$ A
quaternion is defined in the form%
\begin{equation}
q=q_{0}+iq_{1}+jq_{2}+kq_{3}  \tag{1.1}
\end{equation}%
where $q_{0},$ $q_{1},$ $q_{2},$ $q_{3}$ are real numbers and $i,j,k$ are
standard orthonormal basis in $\mathbb{R}^{3}$ which satisfy the quaternion
multiplication rules as%
\begin{eqnarray*}
i^{2} &=&j^{2}=k^{2}=-1 \\
ij &=&-ji=k\text{ \ , \ \ }jk=-kj=i\text{ \ , \ \ }ki=-ik=j.
\end{eqnarray*}%
The quaternion $q$ can be written as 
\begin{equation*}
q=S_{q}+V_{q}
\end{equation*}%
where $S_{q}=q_{0}$ and $V_{q}=iq_{1}+jq_{2}+kq_{3}$. Here $S_{q}$ is called
the scalar part and $V_{q}$ is called the vector part of the quaternion $q$.
If two quaternions are $q=q_{0}+iq_{1}+jq_{2}+kq_{3}$ and $%
p=p_{0}+ip_{1}+jp_{2}+kp_{3}$, then the addition and subtraction of them is 
\begin{equation*}
q\mp p=(q_{0}\mp p_{0})+i(q_{1}\mp p_{1})+j(q_{2}\mp p_{2})+k(q_{3}\mp p_{3})
\end{equation*}%
and the multiplication is%
\begin{eqnarray*}
q.p &=&q_{0}p_{0}-(q_{1}p_{1}+q_{2}p_{2}+q_{3}p_{3}) \\
&&+i\text{ }(q_{0}p_{1}+q_{1}p_{0}+q_{2}p_{3}-q_{3}p_{2}) \\
&&+j\text{ }(q_{0}p_{2}+q_{2}p_{0}+q_{3}p_{1}-q_{1}p_{3}) \\
&&+k\text{ }(q_{0}p_{3}+q_{3}p_{0}+q_{1}p_{2}-q_{2}p_{1}).
\end{eqnarray*}%
The conjugate of the quaternion $q=q_{0}+iq_{1}+jq_{2}+kq_{3}$ is given by $%
\overline{q};$%
\begin{equation*}
\overline{q}=q_{0}-iq_{1}-jq_{2}-kq_{3}.
\end{equation*}%
Also if two quaternions are $q$ and $p$, then%
\begin{eqnarray*}
\overline{\overline{q}} &=&q \\
\overline{(qp)} &=&\overline{p}\text{ }\overline{q}.
\end{eqnarray*}%
The norm of $q$ is defined by $N_{q}$ and 
\begin{equation*}
N_{q}=\left\Vert q\right\Vert =q\overline{q}%
=q_{0}^{2}+q_{1}^{2}+q_{2}^{2}+q_{3}^{2}.
\end{equation*}%
The quaternion $q$ is called a unit quaternion if $N_{q}=1.$

The inverse of $q$ is denoted by $q^{-1}$ as%
\begin{equation*}
q^{-1}=\frac{\overline{q}}{N_{q}}=\frac{\overline{q}}{q\overline{q}}.
\end{equation*}

Since we will clarify the dual Fibonacci quaternions, we now give the basic
notations of dual quaternions.

Clifford \cite{Clif} published his work on dual numbers in 1873 and provided
\ us with a powerfull tool for the analysis of complex numbers. The dual
numbers extend to the real numbers has the form%
\begin{equation*}
d=a+\varepsilon a^{\ast }
\end{equation*}%
where $\varepsilon $ is the dual unit and $\ \varepsilon ^{2}=0$ $,$ $%
\varepsilon \neq 0$. Dual numbers form two dimensional commutative
associative algebra over the real numbers. Also the algebra of dual numbers
is a ring.

A dual quaternion is an extension of dual numbers whereby the elements of
that quaternion \ are dual numbers. Dual quaternions are used as an
appliance for expressing and analyzing the physical properties of rigid
bodies. They are computationally efficient approach of representing rigid
transforms like translation and rotation.

The dual quaternion is represented in the form%
\begin{equation*}
Q=q+\varepsilon q^{\ast }
\end{equation*}%
where $q$ and $q^{\ast }$ are quaternions and $\varepsilon $ is the dual
unit.

If $\ q=q_{0}+iq_{1}+jq_{2}+kq_{3}$ \ and \ $q^{\ast }=q_{0}^{\ast
}+iq_{1}^{\ast }+jq_{2}^{\ast }+kq_{3}^{\ast }$, \ then the dual quaternion $%
Q$ can be denoted as;%
\begin{eqnarray*}
Q &=&q+\varepsilon q^{\ast } \\
&=&(q_{0}+iq_{1}+jq_{2}+kq_{3})+\varepsilon (q_{0}^{\ast }+iq_{1}^{\ast
}+jq_{2}^{\ast }+kq_{3}^{\ast }) \\
&=&(q_{0}+\varepsilon q_{0}^{\ast })+i(q_{1}+\varepsilon q_{1}^{\ast
})+j(q_{2}+\varepsilon q_{2}^{\ast })+k(q_{3}+\varepsilon q_{3}^{\ast }).
\end{eqnarray*}%
So the dual quaternion $Q$ is constructed from eight real parameters. Also $%
Q $ can be written as%
\begin{equation*}
Q=S_{Q}+V_{Q}
\end{equation*}%
where%
\begin{eqnarray*}
S_{Q} &=&q_{0}+\varepsilon q_{0}^{\ast }=S_{q}+\varepsilon S_{q^{\ast }} \\
V_{Q} &=&i(q_{1}+\varepsilon q_{1}^{\ast })+j(q_{2}+\varepsilon q_{2}^{\ast
})+k(q_{3}+\varepsilon q_{3}^{\ast })=V_{q}+\varepsilon V_{q^{\ast }}.
\end{eqnarray*}%
Similarly in quaternions, $S_{Q}$ $\ $and $\ V_{Q}$ are called scalar part
and vector part of the dual quaternion $Q$, respectively.

If two dual quaternions are $Q=q+\varepsilon q^{\ast }$ \ and \ $%
P=p+\varepsilon p^{\ast }$, then the addition and subtraction is given by%
\begin{equation*}
Q\mp P=(q\mp p)+\varepsilon (q^{\ast }\mp p^{\ast })
\end{equation*}%
and multiplication is%
\begin{equation*}
Q.P=q.p+\varepsilon (q.p^{\ast }+q^{\ast }.p)
\end{equation*}%
where $q=q_{0}+iq_{1}+jq_{2}+kq_{3}$ , \ $q^{\ast }=q_{0}^{\ast
}+iq_{1}^{\ast }+jq_{2}^{\ast }+kq_{3}^{\ast }$ \ and \ $%
p=p_{0}+ip_{1}+jp_{2}+kp_{3}$ , \ $p^{\ast }=p_{0}^{\ast }+ip_{1}^{\ast
}+jp+kp_{3}^{\ast }.$

The conjugate of the dual quaternion $Q=q+\varepsilon q^{\ast }$ is defined
as;%
\begin{eqnarray*}
\overline{Q} &=&\overline{q}+\varepsilon \overline{q^{\ast }} \\
&=&(q_{0}+\varepsilon q_{0}^{\ast })-i(q_{1}+\varepsilon q_{1}^{\ast
})-j(q_{2}+\varepsilon q_{2}^{\ast })-k(q_{3}+\varepsilon q_{3}^{\ast }).
\end{eqnarray*}%
If two quaternions are $Q$ \ and $P$, then%
\begin{eqnarray*}
\overline{\overline{Q}} &=&Q \\
\overline{(QP)} &=&\overline{P}\text{ }\overline{Q}.
\end{eqnarray*}%
$N_{Q}$ is called the norm of $Q$ and given by%
\begin{equation*}
N_{Q}=\left\Vert Q\right\Vert =Q\overline{Q}=A^{2}+B^{2}+C^{2}+D^{2}
\end{equation*}%
where $A=q_{0}+\varepsilon q_{0}^{\ast }$ \ , \ $B=q_{1}+\varepsilon
q_{1}^{\ast }$ \ , \ $C=q_{2}+\varepsilon q_{2}^{\ast }$ \ , \ $%
D=q_{3}+\varepsilon q_{3}^{\ast }$. Also the dual number $N_{Q}$ is called
the magnitude of the dual quaternion $Q.$

The dual quaternion with $N_{Q}=1$ is called a unit dual quaternion.

The inverse of $Q$ is%
\begin{equation*}
Q^{-1}=\frac{\overline{Q}}{N_{Q}}=\frac{\overline{Q}}{Q\overline{Q}}.
\end{equation*}%
From these notations it can be said that the above properties are dual
version of quaternions.

There are many works on Fibonacci and Lucas numbers. Dunlap \cite{Dun} ,
Vajda \cite{Vajda}, Verner \cite{Verner} and Hoggatt \cite{Verner} explained
the properties of Fibonacci and Lucas numbers and computed the relations
between them.

Horadam defined the generalized Fibonacci sequences in \cite{Hor}. Then the $%
n^{th}$ Fibonacci and $n^{th}$ Lucas quaternions were described by Horadam
in \cite{Hor2} as%
\begin{equation}
Q_{n}=F_{n}+iF_{n+1}+jF_{n+2}+kF_{n+3}  \tag{1.2}
\end{equation}%
and%
\begin{equation}
K_{n}=L_{n}+iL_{n+1}+jL_{n+2}+kL_{n+3}  \tag{1.3}
\end{equation}%
respectively, where $F_{n}$ is Fibonacci number and $L_{n}$ is Lucas number.
Also $i,j,k$ are standard orthonormal basis in $\mathbb{R}^{3}.$

Swamy \cite{Swamy} gave the relations of generalized Fibonacci quaternions.
Iyer studied Fibonacci quaternions in \cite{Iyer} and obtained some other
relations about Fibonacci and Lucas quaternions.

In \cite{Halici}, Hal\i c\i\ expressed the generating function and Binet
formulas for these quaternions. Akyi\u{g}it, K\"{o}sal and Tosun \cite{Ak}
defined the split Fibonacci and split Lucas quaternions.They also gave Binet
formulas and Cassini identities for these quaternions.

In this paper, we define the dual Fibonacci quaternion and the dual Lucas
quaternion by combining Fibonacci, Lucas quaternions and dual quaternions.
We find the equations between the given quaternions and give the Binet and
Cassini formulas for them.

\section{\textbf{Dual Fibonacci Quaternions}}

Complex Fibonacci numbers are given in \cite{Hor2} by Horadam as;%
\begin{equation*}
C_{n}=F_{n}+iF_{n+1}\text{ \ , \ \ \ }i^{2}=-1
\end{equation*}%
where $F_{n}$ is the $n^{th}$ Fibonacci number.

Also Hal\i c\i\ \cite{Halici2} described the $n^{th}$ complex Fibonacci
quaternion as follows;%
\begin{equation*}
R_{n}=Q_{n}+iQ_{n+1}\text{ \ , \ \ \ }i^{2}=-1
\end{equation*}%
where $Q_{n}=F_{n}+iF_{n+1}+jF_{n+2}+kF_{n+3}$ is the $n^{th}$ Fibonacci
quaternion.

With the same logic we can define dual Fibonacci, Lucas numbers and dual
Fibonacci, Lucas quaternions.

The $n^{th}$ \textit{dual Fibonacci and }$n^{th}$ \textit{dual Lucas numbers}
are defined by 
\begin{eqnarray}
\widetilde{F_{n}} &=&F_{n}+\varepsilon F_{n+1}  \TCItag{2.1} \\
\widetilde{L_{n}} &=&L_{n}+\varepsilon L_{n+1}  \TCItag{2.2}
\end{eqnarray}%
respectively, where $\varepsilon $ is the dual unit and $\ \varepsilon
^{2}=0 $ $,$ $\varepsilon \neq 0.$ Here $F_{n}$ is the $n^{th}$ Fibonacci
number and $L_{n}$ is the $n^{th}$ Lucas number.

The $n^{th}$ \textit{dual Fibonacci quaternion and }$n^{th}$ \textit{dual
Lucas quaternions }are defined as;%
\begin{eqnarray}
\widetilde{Q_{n}} &=&Q_{n}+\varepsilon Q_{n+1}  \TCItag{2.3} \\
\widetilde{K_{n}} &=&K_{n}+\varepsilon K_{n+1}  \TCItag{2.4}
\end{eqnarray}%
respectively. Here $Q_{n}=F_{n}+iF_{n+1}+jF_{n+2}+kF_{n+3}$ is the $n^{th}$
Fibonacci quaternion and $K_{n}=L_{n}+iL_{n+1}+jL_{n+2}+kL_{n+3}$ is the $%
n^{th}$ Lucas quaternion. $i,j,k$ are quaternion units or standard
orthonormal basis in $\mathbb{R}^{3}$ which satisfy the following rules;%
\begin{eqnarray*}
i^{2} &=&j^{2}=k^{2}=-1 \\
ij &=&-ji=k\text{ \ , \ \ }jk=-kj=i\text{ \ , \ \ }ki=-ik=j.
\end{eqnarray*}%
The dual Fibonacci quaternion $\ \widetilde{Q_{n}}$ \ consists four dual
elements and can be represented as%
\begin{equation*}
\widetilde{Q_{n}}=(F_{n}+\varepsilon F_{n+1})+i(F_{n+1}+\varepsilon
F_{n+2})+j(F_{n+2}+\varepsilon F_{n+3})+k(F_{n+3}+\varepsilon F_{n+4})
\end{equation*}%
By using dual Fibonacci numbers we can get%
\begin{equation*}
\widetilde{Q_{n}}=\widetilde{F_{n}}+i\widetilde{F}_{n+1}+j\widetilde{F}%
_{n+2}+k\widetilde{F}_{n+3}.
\end{equation*}%
The scalar part and vector part of the dual Fibonacci quaternion $\widetilde{%
Q_{n}}$ are given by%
\begin{eqnarray*}
S_{\widetilde{Q_{n}}} &=&F_{n}+\varepsilon F_{n+1} \\
V_{\widetilde{Q_{n}}} &=&i(F_{n+1}+\varepsilon
F_{n+2})+j(F_{n+2}+\varepsilon F_{n+3})+k(F_{n+3}+\varepsilon F_{n+4})
\end{eqnarray*}%
respectively.

Let $\widetilde{Q_{n}}=Q_{n}+\varepsilon Q_{n+1}$ \ and \ $\widetilde{P_{n}}%
=P_{n}+\varepsilon P_{n+1}$ \ be two dual Fibonacci quaternions. The
addition and subtraction of them is \ given by%
\begin{equation}
\widetilde{Q_{n}}\mp \widetilde{P_{n}}=(Q_{n}\mp P_{n})+\varepsilon
(Q_{n+1}\mp P_{n+1})  \tag{2.5}
\end{equation}%
and multiplication is 
\begin{equation}
\widetilde{Q_{n}}\widetilde{P_{n}}=Q_{n}P_{n}+\varepsilon
(Q_{n}P_{n+1}+Q_{n+1}P_{n})  \tag{2.6}
\end{equation}%
where $Q_{n}=F_{n}+iF_{n+1}+jF_{n+2}+kF_{n+3}$ \ , \ \ $%
Q_{n+1}=F_{n+1}+iF_{n+2}+jF_{n+3}+kF_{n+4}$ \ and \ $%
P_{n}=X_{n}+iX_{n+1}+jX_{n+2}+kX_{n+3}$ \ , \ \ $%
P_{n+1}=X_{n+1}+iX_{n+2}+jX_{n+3}+kX_{n+4}$ \ are Fibonacci quaternions.

The conjugate of the dual Fibonacci quaternion $\widetilde{Q_{n}}$ is
defined by%
\begin{eqnarray*}
\overline{\widetilde{Q_{n}}} &=&S_{\widetilde{Q_{n}}}-V_{\widetilde{Q_{n}}}
\\
&=&F_{n}+\varepsilon F_{n+1}-i(F_{n+1}+\varepsilon
F_{n+2})-j(F_{n+2}+\varepsilon F_{n+3})-k(F_{n+3}+\varepsilon F_{n+4}).
\end{eqnarray*}%
Also the norm of $\widetilde{Q_{n}}$ can be given as%
\begin{equation*}
N_{\widetilde{Q_{n}}}=\left\Vert \widetilde{Q_{n}}\right\Vert =\widetilde{%
Q_{n}}\overline{\widetilde{Q_{n}}}=A^{2}+B^{2}+C^{2}+D^{2}
\end{equation*}%
where $A=F_{n}+\varepsilon F_{n+1}$ \ , \ $B=F_{n+1}+\varepsilon F_{n+2}$ \
, \ $C=F_{n+2}+\varepsilon F_{n+3}$ \ , \ $D=F_{n+3}+\varepsilon F_{n+4}.$

Additionally, after using the properties of Fibonacci numbers we can write
the norm%
\begin{equation}
N_{\widetilde{Q_{n}}}=F_{2n+1}+F_{2n+5}+2\varepsilon (F_{2n+2}+F_{2n+6}). 
\tag{2.7}
\end{equation}%
If $N_{\widetilde{Q_{n}}}=1$, then $\widetilde{Q_{n}}$ \ is called unit dual
Fibonacci quaternion.

The inverse of $\widetilde{Q_{n}}$ \ can be computed by%
\begin{equation*}
\widetilde{Q_{n}}^{-1}=\frac{\overline{\widetilde{Q_{n}}}}{N_{\widetilde{%
Q_{n}}}}=\frac{\overline{\widetilde{Q_{n}}}}{\widetilde{Q_{n}}\overline{%
\widetilde{Q_{n}}}}.
\end{equation*}%
After above notations now we will state the theorems.

\begin{theorem}
Let $\widetilde{L_{n}}$ and $\widetilde{Q_{n}}$ be a dual Lucas number and a
dual Fibonacci quaternion, respectively. For $n\geq 1$ , the following
relations hold:%
\begin{equation*}
\begin{array}{l}
1)\text{ \ }\widetilde{Q_{n}}+\widetilde{Q}_{n+1}=\widetilde{Q}_{_{n+2}} \\ 
2)\text{ \ }\widetilde{Q_{n}}-i\widetilde{Q}_{n+1}-j\widetilde{Q}_{n+2}-k%
\widetilde{Q}_{_{n+3}}=\widetilde{L}_{n+3} \\ 
3)\text{ \ }\widetilde{Q_{n}}\widetilde{Q_{m}}+\widetilde{Q}_{n+1}\widetilde{%
Q}_{m+1}=-\left( \widetilde{L}_{_{n+m+2}}+\widetilde{L}_{n+m+6}\right) +2\ 
\widetilde{Q}_{n+m+1} \\ 
\text{ \ \ \ \ \ \ \ \ \ \ \ \ \ \ \ \ \ \ \ \ \ \ \ \ \ \ \ \ \ \ \ \ \ \ \ 
}+\varepsilon \left( -L_{n+m+3}-L_{n+m+7}+2Q_{n+m+2}\right)%
\end{array}%
\end{equation*}
\end{theorem}

\begin{proof}
1) By using the equation (2.3) and (2.5), we get%
\begin{eqnarray*}
\widetilde{Q_{n}}+\widetilde{Q}_{n+1} &=&\left( Q_{n}+\varepsilon
Q_{n+1}\right) +\left( Q_{n+1}+\varepsilon Q_{n+2}\right) \\
&=&\left( Q_{n}+Q_{n+1}\right) +\varepsilon (Q_{n+1}+Q_{n+2}).
\end{eqnarray*}%
If we use the equation (1.2) and the identity of Fibonacci numbers that is $%
F_{n}=F_{n-1}+F_{n-2}$ \ , then the above equation becomes as%
\begin{eqnarray*}
\widetilde{Q_{n}}+\widetilde{Q}_{n+1}
&=&F_{n}+F_{n+1}+i(F_{n+1}+F_{n+2})+j(F_{n+2}+F_{n+3})+k(F_{n+3}+F_{n+4}) \\
&&+\varepsilon
(F_{n+1}+F_{n+2}+i(F_{n+2}+F_{n+3})+j(F_{n+3}+F_{n+4})+k(F_{n+4}+F_{n+5})) \\
&=&F_{n+2}+iF_{n+3}+jF_{n+4}+kF_{n+5} \\
&&+\varepsilon (F_{n+3}+iF_{n+4}+jF_{n+5}+kF_{n+6}) \\
&=&Q_{n+2}+\varepsilon Q_{n+3}.
\end{eqnarray*}%
From the definition of dual Fibonacci quaternions it results that%
\begin{equation*}
\widetilde{Q_{n}}+\widetilde{Q}_{n+1}=\widetilde{Q}_{n+2}.
\end{equation*}%
2) From the equation (2.3) and (2.4) , and the identity of Fibonacci
quaternions $Q_{n}-iQ_{n+1}-jQ_{n+2}-kQ_{n+3}=L_{n+3}$ which is given in
Iyer \cite{Iyer}, we find that%
\begin{eqnarray*}
\widetilde{Q_{n}}-i\widetilde{Q}_{n+1}-j\widetilde{Q}_{n+2}-k\widetilde{Q}%
_{n+3} &=&\left( Q_{n}+\varepsilon Q_{n+1}\right) -i\left(
Q_{n+1}+\varepsilon Q_{n+2}\right) \\
&&-j\left( Q_{n+2}+\varepsilon Q_{n+3}\right) -k\left( Q_{n+3}+\varepsilon
Q_{n+4}\right) \\
&=&\left( Q_{n}-iQ_{n+1}-jQ_{n+2}-kQ_{n+3}\right) \\
&&+\varepsilon \left( Q_{n+1}-iQ_{n+2}-jQ_{n+3}-kQ_{n+4}\right) \\
&=&L_{n+3}+\varepsilon L_{n+4} \\
&=&\widetilde{L}_{n+3}.
\end{eqnarray*}%
3) \ By using the equation (2.3) and (2.5), we simply get;%
\begin{eqnarray}
\widetilde{Q_{n}}\widetilde{Q_{m}}+\widetilde{Q}_{n+1}\widetilde{Q}_{m+1}
&=&\left( Q_{n}+\varepsilon Q_{n+1}\right) \left( Q_{m}+\varepsilon
Q_{m+1}\right) +\left( Q_{n+1}+\varepsilon Q_{n+2}\right) \left(
Q_{m+1}+\varepsilon Q_{m+2}\right)  \notag \\
&=&Q_{n}Q_{m}+Q_{n+1}Q_{m+1}  \notag \\
&&+\varepsilon \left[ (Q_{n}Q_{m+1}+Q_{n+1}Q_{m+2})+\left(
Q_{n+1}Q_{m}+Q_{n+2}Q_{m+1}\right) \right] .  \TCItag{2.8}
\end{eqnarray}%
Let us compute $Q_{n}Q_{m}+Q_{n+1}Q_{m+1}$ for Fibonacci quaternions to use
it in the equation by (1.2).%
\begin{eqnarray*}
Q_{n}Q_{m}+Q_{n+1}Q_{m+1} &=&\left( F_{n}+iF_{n+1}+jF_{n+2}+kF_{n+3}\right)
\left( F_{m}+iF_{m+1}+jF_{m+2}+kF_{m+3}\right) \\
&&+\left( F_{n+1}+iF_{n+2}+jF_{n+3}+kF_{n+4}\right) \left(
F_{m+1}+iF_{m+2}+jF_{m+3}+kF_{m+4}\right) \\
&=&F_{n+m+1}-F_{n+m+3}-F_{n+m+5}-F_{n+m+7} \\
&&+\left( 2F_{n+m+2}\right) i+\left( 2F_{n+m+3}\right) j+\left(
2F_{n+m+4}\right) k \\
&=&F_{n+m+1}-F_{n+m+3}-F_{n+m+5}-F_{n+m+7}-2F_{n+m+1} \\
&&+\left( 2F_{n+m+1}\right) +\left( 2F_{n+m+2}\right) i+\left(
2F_{n+m+3}\right) j+\left( 2F_{n+m+4}\right) k \\
&=&-\left( F_{n+m+1}+F_{n+m+3}+F_{n+m+5}+F_{n+m+7}\right) +2Q_{n+m+1} \\
&=&-\left( L_{n+m+2}+L_{n+m+6}\right) +2Q_{n+m+1}.
\end{eqnarray*}%
Thus we can write%
\begin{equation}
Q_{n}Q_{m+1}+Q_{n+1}Q_{m+2}=-\left( L_{n+m+3}+L_{n+m+7}\right) +2Q_{n+m+2} 
\tag{2.9}
\end{equation}%
and%
\begin{equation}
Q_{n+1}Q_{m}+Q_{n+2}Q_{m+1}=-\left( L_{n+m+3}+L_{n+m+7}\right) +2Q_{n+m+2}. 
\tag{2.10}
\end{equation}%
Putting the equations (2.9) and (2.10) in (2.8), we obtain the result as;%
\begin{eqnarray*}
\widetilde{Q_{n}}\widetilde{Q_{m}}+\widetilde{Q}_{n+1}\widetilde{Q}_{m+1}
&=&-\left( L_{n+m+2}+L_{n+m+6}\right) +2Q_{n+m+1} \\
&&+2\varepsilon \left[ -\left( L_{n+m+3}+L_{n+m+7}\right) +2Q_{n+m+2}\right]
\\
&=&-\left( \widetilde{L}_{n+m+2}+\widetilde{L}_{n+m+6}\right) +2\ \widetilde{%
Q}_{n+m+1} \\
&&+\varepsilon \left( -L_{n+m+3}-L_{n+m+7}+2Q_{n+m+2}\right) .
\end{eqnarray*}
\end{proof}

\begin{theorem}
Let $\widetilde{Q_{n}}$ and $\widetilde{K_{n}}$ be a dual Fibonacci
quaternion and a dual Lucas quaternion, respectively. For $n\geq 1$ , the
following relations hold:%
\begin{equation*}
\begin{array}{l}
1)\text{ \ }\widetilde{Q}_{n-1}+\widetilde{Q}_{n+1}=\widetilde{K_{n}} \\ 
2)\text{ \ }\widetilde{Q}_{n+2}-\widetilde{Q}_{n-2}=\widetilde{K_{n}}%
\end{array}%
\end{equation*}
\end{theorem}

\begin{proof}
1) From the equations (2.3) and (2.5), we get%
\begin{eqnarray*}
\widetilde{Q}_{n-1}+\widetilde{Q}_{n+1} &=&\left( Q_{n-1}+\varepsilon
Q_{n}\right) +\left( Q_{n+1}+\varepsilon Q_{n+2}\right) \\
&=&Q_{n-1}+Q_{n+1}+\varepsilon (Q_{n}+Q_{n+2}).
\end{eqnarray*}%
By using the equations (1.2), (1.3) and the relation between Fibonacci
numbers and Lucas numbers $L_{n}=F_{n-1}+F_{n+1}$ \ (see Vajda \cite{Vajda}
), the equation becomes%
\begin{eqnarray*}
\widetilde{Q}_{n-1}+\widetilde{Q}_{n+1}
&=&F_{n-1}+F_{n+1}+i(F_{n}+F_{n+2})+j(F_{n+1}+F_{n+3})+k(F_{n+2}+F_{n+4}) \\
&&+\varepsilon
(F_{n}+F_{n+2}+i(F_{n+1}+F_{n+3})+j(F_{n+2}+F_{n+4})+k(F_{n+3}+F_{n+5})) \\
&=&L_{n}+iL_{n+1}+jL_{n+2}+kL_{n+3} \\
&&+\varepsilon (L_{n+1}+iL_{n+2}+jL_{n+3}+kL_{n+4}) \\
&=&K_{n}+\varepsilon K_{n+1}
\end{eqnarray*}%
If we use the definition of dual Lucas quaternion, the last equation
completed the proof as%
\begin{equation*}
\widetilde{Q}_{n-1}+\widetilde{Q}_{n+1}=\widetilde{K_{n}}.
\end{equation*}%
2) Similarly in \ proof 1), we can find that%
\begin{eqnarray*}
\widetilde{Q}_{n+2}-\widetilde{Q}_{n-2} &=&\left( Q_{n+2}+\varepsilon
Q_{n+3}\right) -\left( Q_{n-2}+\varepsilon Q_{n-1}\right) \\
&=&Q_{n+2}-Q_{n-2}+\varepsilon (Q_{n+3}-Q_{n-1}) \\
&=&F_{n+2}-F_{n-2}+i(F_{n+3}-F_{n-1})+j(F_{n+4}-F_{n})+k(F_{n+5}-F_{n+1}) \\
&&+\varepsilon
(F_{n+3}-F_{n-1}+i(F_{n+4}-F_{n})+j(F_{n+5}-F_{n+1})+k(F_{n+6}-F_{n+2})).
\end{eqnarray*}%
Using the identity $L_{n}=F_{n+2}-F_{n-2}$ \ (see Vajda \cite{Vajda}) and
the equations (1.3), (2.4), it results as%
\begin{equation*}
\widetilde{Q}_{n+2}-\widetilde{Q}_{n-2}=\widetilde{K_{n}}
\end{equation*}
\end{proof}

\begin{theorem}
Let $\widetilde{Q_{n}}$ be a dual Fibonacci quaternion, $\overline{%
\widetilde{Q_{n}}}$ be conjugate of $\widetilde{Q_{n}}$ , $\widetilde{F_{n}}$
be a dual Fibonacci number and $\widetilde{L_{n}}$ be a dual Lucas number.
Then the following equations can be given:%
\begin{equation*}
\begin{array}{l}
1)\text{ \ }\widetilde{Q_{n}}\overline{\widetilde{Q_{n}}}=3(\widetilde{F}%
_{2n+3}+\varepsilon F_{2n+4}) \\ 
\\ 
2)\text{ \ }\widetilde{Q_{n}}+\overline{\widetilde{Q_{n}}}=2\widetilde{F}_{n}
\\ 
\\ 
3)\text{ \ }\widetilde{Q_{n}}^{2}=2\widetilde{Q_{n}}\widetilde{F}_{n}-3(%
\widetilde{F}_{2n+3}+\varepsilon F_{2n+4}) \\ 
\\ 
4)\text{ \ }\widetilde{Q_{n}}\overline{\widetilde{Q_{n}}}+\text{\ }%
\widetilde{Q}_{n-1}\overline{\widetilde{Q}}_{n-1}=3(\widetilde{L}%
_{2n+2}+\varepsilon L_{2n+3}) \\ 
\\ 
5)\text{ \ }\widetilde{Q_{n}}^{2}+\widetilde{Q}_{n-1}^{2}=2\widetilde{Q}%
_{2n-1}-3\widetilde{L}_{2n+2}+\varepsilon (2Q_{2n}-3L_{2n+3})%
\end{array}%
\end{equation*}
\end{theorem}

\begin{proof}
1) \ From the equation (2.1) , (2.7) and the identitiy $%
F_{n}=F_{n-1}+F_{n-2} $, we clearly get%
\begin{eqnarray*}
\widetilde{Q_{n}}\overline{\widetilde{Q_{n}}} &=&F_{2n+1}+F_{2n+5}+2%
\varepsilon (F_{2n+2}+F_{2n+6}) \\
&=&3(F_{2n+3}+2\varepsilon F_{2n+4}) \\
&=&3(F_{2n+3}+\varepsilon F_{2n+4}+\varepsilon F_{2n+4}) \\
&=&3(\widetilde{F}_{2n+3}+\varepsilon F_{2n+4}).
\end{eqnarray*}%
2) \ By using the equations (2.3), (2.5) and (2.7) \ we can compute;%
\begin{eqnarray*}
\widetilde{Q_{n}}+\overline{\widetilde{Q_{n}}} &=&2(F_{n}+\varepsilon
F_{n+1}) \\
&=&2\widetilde{F}_{n}.
\end{eqnarray*}%
3) \ We obtain the result by using the equations given in theorem 3
(identitiy 1) and 2)) as;%
\begin{eqnarray*}
\widetilde{Q_{n}}^{2} &=&\widetilde{Q_{n}}\widetilde{Q_{n}} \\
&=&\widetilde{Q_{n}}(2\widetilde{F}_{n}-\overline{\widetilde{Q_{n}}}) \\
&=&2\widetilde{Q_{n}}\widetilde{F}_{n}-\widetilde{Q_{n}}\overline{\widetilde{%
Q_{n}}} \\
&=&2\widetilde{Q_{n}}\widetilde{F}_{n}-3(\widetilde{F}_{2n+3}+\varepsilon
F_{2n+4}).
\end{eqnarray*}%
4) \ If we consider identity 1) in this theorem and $L_{n}=F_{n-1}+F_{n+1}$
\ (see Vajda \cite{Vajda} ), we can obtain the result;%
\begin{eqnarray*}
\widetilde{Q_{n}}\overline{\widetilde{Q_{n}}}+\text{\ }\widetilde{Q}_{n-1}%
\overline{\widetilde{Q}}_{n-1} &=&3(\widetilde{F}_{2n+3}+\varepsilon
F_{2n+4})+3(\widetilde{F}_{2n+1}+\varepsilon F_{2n+2}) \\
&=&3(\widetilde{F}_{2n+1}+\widetilde{F}_{2n+3}+\varepsilon
(F_{2n+2}+F_{2n+4})) \\
&=&3(\widetilde{L}_{2n+2}+\varepsilon L_{2n+3}).
\end{eqnarray*}%
5) \ From the equations (2.3) and (2.5), we have;%
\begin{equation}
\widetilde{Q_{n}}^{2}+\widetilde{Q}_{n-1}^{2}=Q_{n}^{2}+Q_{n-1}^{2}+2%
\varepsilon \left( Q_{n-1}Q_{n}+Q_{n}Q_{n+1}\right) .  \tag{2.11}
\end{equation}%
Here if we use the identity $Q_{n}^{2}+Q_{n-1}^{2}=2Q_{2n-1}-3L_{2n+2}$ (see
Swamy \cite{Swamy} ), we find;%
\begin{equation*}
\widetilde{Q_{n}}^{2}+\widetilde{Q}_{n-1}^{2}=2Q_{2n-1}-3L_{2n+2}+2%
\varepsilon \left( Q_{n-1}Q_{n}+Q_{n}Q_{n+1}\right) .
\end{equation*}%
Now by using $F_{n}F_{m}+F_{n+1}F_{m+1}=F_{n+m+1}$ (Vajda \cite{Vajda} ) and 
$F_{n}^{2}+F_{n+1}^{2}=F_{2n+1}$ (Vajda \cite{Vajda} ) , we get the
following equation;%
\begin{eqnarray*}
Q_{n-1}Q_{n}+Q_{n}Q_{n+1} &=&F_{n-1}F_{n}-F_{n+3}F_{n+4}-2F_{2n+4}+i\left(
2F_{2n+1}\right) \\
&&+j\left( 2F_{2n+2}\right) +k\left( 2F_{2n+3}\right) .
\end{eqnarray*}%
In this equation, we take into account that \ $F_{n}=F_{n-1}+F_{n-2}$ \ and
(1.2), thus we have;%
\begin{eqnarray*}
Q_{n-1}Q_{n}+Q_{n}Q_{n+1} &=&F_{n-1}F_{n}-F_{n+3}\left(
F_{n+3}+F_{n+2}\right) -2F_{2n+4}-2F_{2n}+2F_{2n} \\
&&+i\left( 2F_{2n+1}\right) +j\left( 2F_{2n+2}\right) +k\left(
2F_{2n+3}\right) \\
&=&F_{n-1}F_{n}-F_{n+3}^{2}-F_{n+3}F_{n+2} \\
&&-2\left( F_{2n+4}+F_{2n}\right) +2\left(
F_{2n}+iF_{2n+1}+jF_{2n+2}+kF_{2n+3}\right) \\
&=&F_{n-1}F_{n}-F_{n+3}^{2}-F_{n+2}^{2}-F_{n+1}^{2}-F_{n}^{2}-F_{n-1}F_{n} \\
&&-2\left( 3F_{2n+2}\right) +2\left( Q_{2n}\right) \\
&=&-\left( F_{n+3}^{2}+F_{n+2}^{2}\right) -\left(
F_{n+1}^{2}+F_{n}^{2}\right) -2\left( 3F_{2n+2}\right) +2\left( Q_{2n}\right)
\\
&=&-\left( F_{2n+5}+F_{n+1}\right) -2\left( 3F_{2n+2}\right) +2\left(
Q_{2n}\right) \\
&=&-3F_{2n+3}-6F_{2n+2}+2Q_{2n}
\end{eqnarray*}%
Putting this equation in (2.11) \ , using the identity $%
L_{n}=F_{n-1}+F_{n+1} $, the equations (2.3), (2.4) and composing the
expression , the result is obtained as;%
\begin{eqnarray*}
\widetilde{Q_{n}}^{2}+\widetilde{Q}_{n-1}^{2}
&=&2Q_{2n-1}-3L_{2n+2}+2\varepsilon \left(
-3F_{2n+3}-6F_{2n+2}+2Q_{2n}\right) \\
&=&2\left( Q_{2n-1}+\varepsilon Q_{2n}\right) +\varepsilon Q_{2n}-3L_{2n+2}
\\
&&+2\varepsilon \left[ -3\left( F_{2n+3}+2F_{2n+2}\right) +Q_{2n}\right] \\
&=&2\widetilde{Q}_{2n-1}-3\left( L_{2n+2}+\varepsilon L_{2n+3}\right)
+\varepsilon \left( 2Q_{2n}-3L_{2n+3}\right) \\
&=&\widetilde{Q}_{2n-1}-3\widetilde{L}_{2n+2}+\varepsilon \left(
2Q_{2n}-3L_{2n+3}\right) .
\end{eqnarray*}
\end{proof}

\begin{theorem}
Let $\widetilde{Q_{n}}$ be a dual Fibonacci quaternion. Then the following
summation formulas hold:%
\begin{equation*}
\begin{array}{l}
1)\text{ \ }\overset{n}{\underset{s=1}{\dsum }}\ \widetilde{Q_{s}}=\ 
\widetilde{Q}_{s+2}-\ \widetilde{Q_{2}} \\ 
\\ 
2)\text{ \ }\left( \overset{n}{\underset{s=0}{\dsum }}\ \widetilde{Q}%
_{n+s}\right) +\widetilde{Q}_{_{n+1}}=\ \widetilde{Q}_{n+p+2} \\ 
\\ 
3)\text{ \ }\overset{n}{\underset{s=1}{\dsum }}\ \widetilde{Q}_{2s-1}=\ 
\widetilde{Q_{2n}}-\widetilde{Q_{0}} \\ 
\\ 
4)\text{ \ }\overset{n}{\underset{s=1}{\dsum }}\ \widetilde{Q_{2s}}=\ 
\widetilde{Q}_{2n+1}-\widetilde{Q_{1}}%
\end{array}%
\end{equation*}
\end{theorem}

\begin{proof}
1) \ If we use the equation (2.3) and $\overset{n}{\underset{s=1}{\dsum }}%
Q_{s}=Q_{n+2}-Q_{2}$ (see Hal\i c\i ,\ \cite{Halici} ), we obtain the result;%
\begin{eqnarray*}
\overset{n}{\underset{s=1}{\dsum }}\ \widetilde{Q_{s}} &=&\overset{n}{%
\underset{s=1}{\dsum }}\ \left( Q_{s}+\varepsilon Q_{s+1}\right) \\
&=&\overset{n}{\underset{s=1}{\dsum }}Q_{s}+\varepsilon \overset{n}{\underset%
{s=1}{\dsum }}Q_{s+1} \\
&=&Q_{n+2}-Q_{2}+\varepsilon \left( Q_{n+3}-Q_{3}\right) \\
&=&\widetilde{Q}_{n+2}-\widetilde{Q_{2}}.
\end{eqnarray*}%
2) \ Since 
\begin{equation*}
\overset{n}{\underset{s=0}{\dsum }}Q_{n+s}=\overset{n+p}{\underset{r=1}{%
\dsum }}Q_{r}-\overset{n-1}{\underset{r=1}{\dsum }}Q_{r}
\end{equation*}%
then we can write;%
\begin{eqnarray*}
\left( \overset{n}{\underset{s=0}{\dsum }}\ \widetilde{Q}_{n+s}\right) +%
\widetilde{Q}_{n+1} &=&\left( \overset{n}{\underset{s=0}{\dsum }}%
Q_{n+s}+\varepsilon Q_{n+s+1}\right) +\left( Q_{n+1}+\varepsilon
Q_{n+2}\right) \\
&=&\left[ \left( \overset{n}{\underset{s=0}{\dsum }}Q_{n+s}\right) +Q_{n+1}%
\right] +\varepsilon \left[ \left( \overset{n}{\underset{s=0}{\dsum }}%
Q_{n+s+1}\right) +Q_{n+2}\right] \\
&=&\left[ \left( \overset{n+p}{\underset{r=1}{\dsum }}Q_{r}-\overset{n-1}{%
\underset{r=1}{\dsum }}Q_{r}\right) +Q_{n+1}\right] \\
&&+\varepsilon \left[ \left( \overset{n+p+1}{\underset{r=1}{\dsum }}Q_{r}-%
\overset{n}{\underset{r=1}{\dsum }}Q_{r}\right) +Q_{n+2}\right] \\
&=&Q_{n+p+2}+\varepsilon Q_{n+p+3} \\
&=&\widetilde{Q}_{n+p+2}.
\end{eqnarray*}%
3) \ Using (2.3), we obtain the result clearly;%
\begin{eqnarray*}
\overset{n}{\underset{s=1}{\dsum }}\ \widetilde{Q}_{2s-1} &=&\overset{n}{%
\underset{s=1}{\dsum }}\left( Q_{2s-1}+\varepsilon Q_{2s}\right) \\
&=&\overset{n}{\underset{s=1}{\dsum }}\left( Q_{2s}-Q_{2s-2}\right)
+\varepsilon \overset{n}{\underset{s=1}{\dsum }}\left(
Q_{2s+1}-Q_{2s-1}\right) \\
&=&Q_{2n}-Q_{0}+\varepsilon \left( Q_{2n+1}-Q_{1}\right) \\
&=&\widetilde{Q_{2n}}-\widetilde{Q_{0}}.
\end{eqnarray*}%
4) \ Similarly in proof 3);%
\begin{eqnarray*}
\overset{n}{\underset{s=1}{\dsum }}\ \widetilde{Q_{2s}} &=&\overset{n}{%
\underset{s=1}{\dsum }}\left( Q_{2s}+\varepsilon Q_{2s+1}\right) \\
&=&\overset{n}{\underset{s=1}{\dsum }}\left( Q_{2s+1}-Q_{2s-1}\right)
+\varepsilon \overset{n}{\underset{s=1}{\dsum }}\left( Q_{2s+2}-Q_{2s}\right)
\\
&=&Q_{2n+1}-Q_{1}+\varepsilon \left( Q_{2n+2}-Q_{2}\right) \\
&=&\widetilde{Q}_{2n+1}-\widetilde{Q_{1}}
\end{eqnarray*}
\end{proof}

The explicit formulas for Fibonacci and Lucas numbers were given by
Jacques-Phillipe-Marie Binet in 1843, which are called Binet formulas. The
positive and negative roots of the quadratic equation $x^{2}-x-1=0$ \ are $%
\alpha $ and $\beta $, respectively. They are;%
\begin{equation*}
\alpha =\frac{1+\sqrt{5}}{2}\text{ \ \ \ and \ \ \ }\beta =\frac{1-\sqrt{5}}{%
2}.
\end{equation*}%
Then the Binet formulas for Fibonacci and Lucas numbers are given by (see
Koshy, \cite{Koshy})%
\begin{equation*}
F_{n}=\frac{\alpha ^{n}-\beta ^{n}}{\alpha -\beta }\text{ \ \ \ \ and \ \ \
\ }L_{n}=\alpha ^{n}+\beta ^{n}.
\end{equation*}%
Now we will express the following theorems which are about Binet formula and
Cassini identitiy.

\begin{theorem}
Let $\widetilde{Q_{n}}$ and $\widetilde{K_{n}}$ be dual Fibonacci and dual
Lucas quaternions, respectively. For $n\geq 0$, the Binet formulas for these
quaternions are given as;%
\begin{equation*}
\widetilde{Q_{n}}=\dfrac{\alpha ^{\ast }\alpha ^{n}-\beta ^{\ast }\beta ^{n}%
}{\alpha -\beta }
\end{equation*}%
and%
\begin{equation*}
\widetilde{K_{n}}=\alpha ^{\ast }\alpha ^{n}+\beta ^{\ast }\beta ^{n}
\end{equation*}%
where $\alpha ^{\ast }=\underline{\alpha }\left( 1+\varepsilon \alpha
\right) $ \ ,\ $\beta ^{\ast }=\underline{\beta }\left( 1+\varepsilon \beta
\right) $ \ and $\ \underline{\alpha }=1+i\alpha +j\alpha ^{2}+k\alpha ^{3}$
\ , \ $\underline{\beta }=1+i\beta +j\beta ^{2}+k\beta ^{3}$.
\end{theorem}

\begin{proof}
In \cite{Halici}, Hal\i c\i\ gave the Binet formula for Fibonacci quaternion
by%
\begin{equation}
Q_{n}=\dfrac{\underline{\alpha }\alpha ^{n}-\underline{\beta }\beta ^{n}}{%
\alpha -\beta }  \tag{2.12}
\end{equation}%
where $\alpha =\dfrac{1+\sqrt{5}}{2}$ \ , \ $\beta =\dfrac{1-\sqrt{5}}{2}$ \
and $\underline{\alpha }=1+i\alpha +j\alpha ^{2}+k\alpha ^{3}$ , \ $%
\underline{\beta }=1+i\beta +j\beta ^{2}+k\beta ^{3}.$ Thus it can be
written;%
\begin{equation}
Q_{n+1}=\dfrac{\underline{\alpha }\alpha ^{n+1}-\underline{\beta }\beta
^{n+1}}{\alpha -\beta }.  \tag{2.13}
\end{equation}%
So by using the equations (2.3), (2.12) and (2.13), we have;%
\begin{eqnarray*}
\widetilde{Q_{n}} &=&Q_{n}+\varepsilon Q_{n+1} \\
&=&\dfrac{\underline{\alpha }\alpha ^{n}-\underline{\beta }\beta ^{n}}{%
\alpha -\beta }+\varepsilon \text{ }\dfrac{\underline{\alpha }\alpha ^{n+1}-%
\underline{\beta }\beta ^{n+1}}{\alpha -\beta } \\
&=&\frac{\underline{\alpha }\alpha ^{n}\left( 1+\varepsilon \alpha \right) -%
\underline{\beta }\beta ^{n}\left( 1+\varepsilon \beta \right) }{\alpha
-\beta }.
\end{eqnarray*}%
Taking $\underline{\alpha }\left( 1+\varepsilon \alpha \right) $ $=\alpha
^{\ast }$ \ and $\underline{\beta }\left( 1+\varepsilon \beta \right) $ $%
=\beta ^{\ast }$ in last equation, then the proof completed as;%
\begin{equation*}
\widetilde{Q_{n}}=\dfrac{\alpha ^{\ast }\alpha ^{n}-\beta ^{\ast }\beta ^{n}%
}{\alpha -\beta }.
\end{equation*}%
Similarly in \cite{Halici}, the Binet formula for Lucas quaternion is given
by%
\begin{equation*}
T_{n}=\underline{\alpha }\alpha ^{n}+\underline{\beta }\beta ^{n}.
\end{equation*}%
We can clearly compute that;%
\begin{eqnarray*}
\widetilde{K_{n}} &=&K_{n}+\varepsilon K_{n+1} \\
&=&\underline{\alpha }\alpha ^{n}+\underline{\beta }\beta ^{n}+\varepsilon
\left( \underline{\alpha }\alpha ^{n+1}+\underline{\beta }\beta ^{n+1}\right)
\\
&=&\alpha ^{\ast }\alpha ^{n}+\beta ^{\ast }\beta ^{n}.
\end{eqnarray*}
\end{proof}

\begin{theorem}
Let $\widetilde{Q_{n}}$ and $\widetilde{K_{n}}$ be dual Fibonacci and dual
Lucas quaternions, respectively. For $n\geq 1$, the Cassini identities for
these quaternions are given as;%
\begin{equation*}
\widetilde{Q}_{n-1}\widetilde{Q}_{n+1}-\widetilde{Q_{n}}^{2}=\left(
-1\right) ^{n}\left[ 2\widetilde{Q_{1}}-3k-\varepsilon \text{ }9k\right]
\end{equation*}%
and 
\begin{equation*}
\widetilde{K}_{n-1}\widetilde{K}_{n+1}-\widetilde{K_{n}}^{2}=5\left(
-1\right) ^{n}\left[ 2\widetilde{K_{1}}-4k+\varepsilon \left( -2i-17k\right) %
\right]
\end{equation*}
\end{theorem}

\begin{proof}
From the equations (2.3), (2.5) and (2.6), we have ;%
\begin{equation}
\widetilde{Q}_{n-1}\widetilde{Q}_{n+1}-\widetilde{Q_{n}}^{2}=\left(
Q_{n-1}Q_{n+1}-Q_{n}^{2}\right) +\varepsilon \left(
Q_{n-1}Q_{n+2}-Q_{n}Q_{n+1}\right) .  \tag{2.14}
\end{equation}%
If we use the equation for Fibonacci numbers called D'ocagne's identity
which is $F_{m}F_{n+1}-F_{m+1}F_{n}=\left( -1\right) ^{n}F_{m-n}$ (see
Weisstein,\cite{Weis} ) and the identitiy of negafibonacci numbers $%
F_{-n}=\left( -1\right) ^{n+1}F_{n}$ (see Knuth, \cite{Knuth} ), we get%
\begin{equation*}
Q_{n-1}Q_{n+2}-Q_{n}Q_{n+1}=(-1)^{n}(2+4i++6j+k).
\end{equation*}%
Also Hal\i c\i\ denoted the equation in \cite{Halici} as;%
\begin{equation*}
Q_{n-1}Q_{n+1}-Q_{n}^{2}=\left( -1\right) ^{n}\left( 2Q_{1}-3k\right) .
\end{equation*}%
Putting these last two equations in (2.14) and using the definition of
Fibonacci quaternion, we reach the result;%
\begin{eqnarray*}
\widetilde{Q}_{n-1}\widetilde{Q}_{n+1}-\widetilde{Q_{n}}^{2} &=&\left(
Q_{n-1}Q_{n+1}-Q_{n}^{2}\right) +\varepsilon \left(
Q_{n-1}Q_{n+2}-Q_{n}Q_{n+1}\right) \\
&=&\left( -1\right) ^{n}\left( 2Q_{1}-3k\right) +\varepsilon
((-1)^{n}(2+4i++6j+k)) \\
&=&\left( -1\right) ^{n}(2\widetilde{Q_{1}}-3k-\varepsilon 9k).
\end{eqnarray*}%
By the equation (2.4), the identity of Lucas numbers which are $%
L_{n-1}L_{n+1}-L_{n}^{2}=5(-1)^{n-1}$ (see Koshy, \cite{Koshy}) and $%
L_{n+2}=L_{n+1}+L_{n}$ (see Dunlap,\cite{Dun}), the proof completed as;%
\begin{equation*}
\widetilde{K}_{n-1}\widetilde{K}_{n+1}-\widetilde{K_{n}}^{2}=5\left(
-1\right) ^{n}\left[ 2\widetilde{K_{1}}-4k+\varepsilon \left( -2i-17k\right) %
\right] .
\end{equation*}
\end{proof}

\bigskip

Semra Kaya Nurkan

University of U\c{s}ak

Department of Mathematics

64200, U\c{s}ak, Turkey

E-mail: semra.kaya@usak.edu.tr, semrakaya\_gs@yahoo.com

\bigskip

\.{I}lkay Arslan G\"{u}ven

University of Gaziantep

Department of Mathematics

\c{S}ehitkamil, 27310, Gaziantep, Turkey

E-mail: iarslan@gantep.edu.tr, ilkayarslan81@hotmail.com


\begin{thebibliography}{99}
\bibitem{Ak} M. Akyi\u{g}it, H.H. K\"{o}sal and M. Tosun, \textit{Split
Fibonacci Quaternions}. Adv. in Appl. Clifford Algebras, 23 (2013), 535-545.

\bibitem{Clif} W. K. Clifford, \textit{Preliminary Sketch of Bi-quaternions.}
Proc. London Math. Soc., 4 (1873), 381-395.

\bibitem{Dun} R. A. Dunlap, \textit{The Golden Ratio and Fibonacci Numbers.}
World Scientific, 1997.

\bibitem{Halici} S. Hal\i c\i , \textit{On Fibonacci Quaternions.} Adv. in
Appl. Clifford Algebras, 22 (2012), 321-327.

\bibitem{Halici2} S. Hal\i c\i , \textit{On Complex Fibonacci Quaternions.}
Adv. in Appl. Clifford Algebras, 23 (2013), 105-112.

\bibitem{Hor} A. F. Horadam, \textit{A Generalized Fibonacci Sequence.}
American Math. Monthly, 68 (1961), 455-459.

\bibitem{Hor2} A. F. Horadam, \textit{Complex Fibonacci Numbers and
Fibonacci Quaternions.} American Math. Monthly, 70 (1963), 289-291.

\bibitem{Iyer} M. R. Iyer, \textit{Some Results on \ Fibonacci Quaternions. }%
The Fibonacci Quarterly, 7(2) (1969), 201-210.

\bibitem{Iyer2} M. R. Iyer, \textit{A Note on Fibonacci Quaternion. }The
Fibonacci Quarterly, 3 (1969), 225-229.

\bibitem{Knuth} D. Knuth, \textit{Negafibonacci Numbers and Hyperbolic
Plane. }Annual Meeting of the Math. Association of America, 2013.

\bibitem{Koshy} T. Koshy, \textit{Fibonacci and Lucas Numbers with
Applications. }A Wiley-Intersience Publication, USA, 2001.

\bibitem{Swamy} M. N. Swamy, \textit{On Generalized Fibonacci Quaternions. }%
The Fibonacci Quarterly, 5 (1973), 547-550.

\bibitem{Vajda} S. Vajda, \textit{Fibonacci and Lucas Numbers and the Golden
Section. }Ellis Horwood Limited Publ., England, 1989.

\bibitem{Verner} E. Verner and Jr. Hoggatt, \textit{Fibonacci and Lucas
Numbers. }The Fibonacci Association, 1969.

\bibitem{Weis} E. W. Weisstein, \textit{Fibonacci Number. }MathWorld (online
mathematics reference work).
\end{thebibliography}
\end{document}